\documentclass[11pt]{amsart}
\usepackage[T1]{fontenc}
\usepackage{lmodern}
\usepackage{amsmath,amssymb,amsthm,mathtools}
\usepackage{enumitem,microtype,needspace}
\usepackage[margin=1.1in]{geometry}
\usepackage[colorlinks=true,linkcolor=blue,citecolor=blue,urlcolor=blue]{hyperref}

\usepackage{tikz}
\usetikzlibrary{arrows.meta,decorations.markings,calc}

\newtheorem{theorem}{Theorem}[section]
\newtheorem{proposition}[theorem]{Proposition}
\newtheorem{lemma}[theorem]{Lemma}
\newtheorem{corollary}[theorem]{Corollary}
\theoremstyle{definition}
\newtheorem{definition}[theorem]{Definition}
\newtheorem{question}[theorem]{Question}
\theoremstyle{remark}
\newtheorem{remark}[theorem]{Remark}

\newcommand{\ThZ}{\Theta^3_{\mathbb{Z}}}
\newcommand{\shsim}{\sim_{\mathrm{sh}}}
\numberwithin{equation}{section}
\setlist[enumerate]{leftmargin=*,label=\textup{(\roman*)},itemsep=3pt,topsep=5pt}
\title[Shake concordance and the Hopf link]{Shake concordance in the topological concordance class of the Hopf link}

\author[W.\ Shen]{Weizhe Shen}
\address{Department of Mathematics, University of California, Davis}
\email{wzshen@ucdavis.edu}

\begin{document}

\begin{abstract}
    Within the topological concordance class of the Hopf link, we construct infinitely many classes under the equivalence relation generated by smooth $0$-shake concordance. Each contains infinitely many smooth concordance classes represented by links that are pairwise strongly $0$-shake concordant and whose second components have unbounded smooth slice genus.
\end{abstract}

\maketitle

% \tableofcontents

%----------------------------------------------------%
%----------------------------------------------------%
\section{Introduction}
The topological concordance class of the Hopf link contains links that are not smoothly concordant to it \cite{CKRS,DR,KKP}. Shake concordance gives a weaker smooth relation than smooth concordance: every knot together with a positive meridian is $0$-shake concordant to the Hopf link \cite[Proposition 3.1]{Bos19}, and the Cha--Kim--Ruberman--Strle examples are even strongly $0$-shake concordant to it \cite[Proposition 3.5]{Bos19}. These results leave open how much of the smooth concordance structure survives when passing to shake concordance. We construct infinitely many shake-equivalence classes, each containing infinitely many smooth concordance classes represented by links that are topologically concordant to the Hopf link and whose second components have unbounded smooth slice genus.

We use \textit{shake concordance} to mean smooth $0$-shake concordance. For ordered, oriented links $L$ and $L'$ in $S^3$ with the same number of components, we write $L\shsim L'$ and call $L$ and $L'$ \textit{shake equivalent} if they are joined by a finite chain of shake concordances. Thus $\shsim$ is the equivalence relation generated by shake concordance, and a \textit{shake-equivalence class} is an equivalence class for $\shsim$. We do not assume that shake concordance itself is transitive. See Section \ref{sec:shake} for the link definitions and \cite[Section 2 and Corollary 3.8]{CR} for the corresponding distinction in the knot setting.

For a link $L\subset S^3$, let $M_L$ denote the closed oriented $3$-manifold obtained by zero-framed surgery on all components of $L$. For knots $J$ and $K$, let $\mathcal{H}(J,K)$ be the ordered link obtained from the positive Hopf link by tying $J$ and $K$ into its first and second components, respectively, in disjoint balls. Write $U$ for the unknot, $T$ for the right-handed trefoil, and $H=\mathcal{H}(U,U)$ for the Hopf link. We use Hedden's convention for the positive-clasped, $t$-twisted Whitehead double $D_+(K,t)$ \cite{Hed}, and abbreviate the untwisted operation by
\[
    \mathcal{D}(K)=D_+(K,0).
\]
Its iterates are denoted by $\mathcal{D}^j$. Set
\begin{equation}\label{eq:family}
    K_n=D_+(U,-n),\qquad K_\ast=\mathcal{D}(T),\qquad Z=\mathcal{D}^2(K_\ast)=\mathcal{D}^3(T),\qquad n\geq1.
\end{equation}
In particular, $K_1=T$. The group of oriented integral homology $3$-spheres modulo smooth integral homology cobordism is denoted by $\ThZ$.

\Needspace{16\baselineskip}
\begin{theorem}\label{thm:main}
    For each $n\geq 1$, let
    \[
        L_n=\mathcal{H}\bigl(\mathcal{D}(K_n),K_\ast\bigr).
    \]
    Then $L_n$ is topologically concordant to the Hopf link, and there is an orientation-preserving diffeomorphism
    \begin{equation}\label{eq:main-surgery}
        M_{L_n}\cong S^3_{1/n}(Z).
    \end{equation}
    The classes $[M_{L_n}]$, $n\geq 1$, are linearly independent in $\ThZ$, whereas their Ozsv\'ath--Szab\'o correction terms satisfy
    \[
        d(M_{L_n})=-2\qquad\text{for every }n\geq 1.
    \]
    Consequently, $L_n\not\shsim L_m$ for $n\neq m$, and $L_n\not\shsim H$ for every $n$. The non-equivalence conclusions remain valid if component orderings and orientations are disregarded.
\end{theorem}

The first member of the family is $L_1=\mathcal{H}(K_\ast,K_\ast)$. Thus tying the same positive untwisted Whitehead double into both Hopf components already gives a link that is topologically concordant, but not shake equivalent, to the Hopf link. The correction term alone separates every $L_n$ from $H$, since $d(M_{L_n})=-2$ and $d(M_H)=0$. The linear-independence statement distinguishes the $L_n$ from one another, even though their correction terms are equal.

The construction is based on two surgery identities. The first turns the twisting parameter $-n$ into the surgery coefficient $1/n$. The second transfers an untwisted Whitehead double between the two local knots while preserving the oriented zero-surgery manifold.

\begin{theorem}\label{thm:surgery}
    For knots $J,K$ and a nonzero integer $n$, there are orientation-preserving diffeomorphisms
    \begin{align}
        M_{\mathcal{H}(K,D_+(U,-n))}
        &\cong S^3_{1/n}\bigl(\mathcal{D}(K)\bigr),\label{eq:twist}\\
        M_{\mathcal{H}(\mathcal{D}(J),K)}
        &\cong M_{\mathcal{H}(J,\mathcal{D}(K))}.\label{eq:transfer}
    \end{align}
\end{theorem}

The cases $n=1$ and $n=-1$ of \eqref{eq:twist} recover, respectively, the trefoil and figure-eight surgery formulas recorded by Tange \cite[Section 1.1]{Tan}, with special cases in Tsuchiya's work \cite[Theorem 1.1]{Tsu}. Both identities in Theorem \ref{thm:surgery} follow from the peripheral structure and classical symmetries of the Whitehead link. We give proofs in Section \ref{sec:whitehead}, including the orientation conventions needed to determine the sign of $1/n$.

For links with unimodular zero-framed linking matrices, the sphere-surgery construction of \cite{CR,Bos21} makes zero-surgery invariant under shake concordance up to smooth integral homology cobordism; Section \ref{sec:shake} gives the homology calculation. Theorem \ref{thm:surgery} identifies the $M_{L_n}$ with surgeries whose classes are linearly independent in $\ThZ$ by a theorem of Baldwin and Sivek \cite[Corollary 1.5]{BS}, building on \cite{NST}. The local knots of each $L_n$ have Alexander polynomial one and are therefore topologically slice by Freedman's theorem \cite[Theorem 11.7B]{FQ}, giving the topological concordances to the Hopf link.

Our second construction produces infinitely many smooth concordance classes within each of these shake-equivalence classes.

\begin{theorem}\label{thm:refinement}
    For every $n\geq 1$, there is a sequence of ordered two-component links
    \[
        \{\widehat L_{n,j}\}_{j\geq 0},\qquad \widehat L_{n,0}=L_n,
    \]
    all topologically concordant to $H$, such that any two members admit a single strong shake concordance. No two members are smoothly concordant, even after component orderings and orientations are disregarded. The first component of every $\widehat L_{n,j}$ has knot type $\mathcal{D}(K_n)$; writing $C_{n,j}$ for the second component of $\widehat L_{n,j}$, we have
    \[
        g_4(C_{n,j})\geq j+1\qquad(j\geq 0),
    \]
    where $g_4$ denotes the smooth slice genus. In particular, these genera tend to infinity as $j\to\infty$.
\end{theorem}

The proof combines a relative form of the satellite construction of Cochran and Ray \cite[Proposition 3.1]{CR} with a pattern of Davis and Ray \cite{DR}. For each pair of links, we apply the relative construction to obtain a single strong shake concordance; no transitivity assumption is needed.

\vspace{6pt}
\paragraph{Organization}
Section \ref{sec:shake} recalls the definition of $0$-shake concordance and establishes the homology cobordism obstruction arising from $0$-surgery. Section \ref{sec:whitehead} proves the Whitehead-link surgery identities of Theorem \ref{thm:surgery}. Section \ref{sec:floer} combines these identities with Floer-theoretic results to prove Theorem \ref{thm:main}. Finally, Section \ref{sec:examples} gives further examples, proves Theorem \ref{thm:refinement}, and discusses converse questions and nonzero framings.

\vspace{6pt}
\paragraph{Conventions}
All knots and links lie in the oriented $3$-sphere. Concordances and cobordisms are oriented and, unless specified as topological, smooth. Topological concordances are locally flat. Links are ordered and oriented unless these choices are explicitly disregarded. All ordinary homology groups have integral coefficients. Surgery manifolds carry the orientation induced from $S^3$, and $\cong$ between oriented manifolds denotes an orientation-preserving diffeomorphism. The construction of $\mathcal{H}(J,K)$ uses the positive Hopf clasp; the corresponding splice convention is recorded in Lemma \ref{lem:splice}.

\section*{Acknowledgments}
\noindent The author is grateful to Jen Hom for her guidance during graduate school and for many discussions that have informed this work. The author also thanks Aru Ray for a helpful conversation on shake concordance.

%----------------------------------------------------%
%----------------------------------------------------%
\section{Shake concordance and zero-surgery}\label{sec:shake}

\begin{definition}[{Bosman \cite[Definition 2.3]{Bos19}}]\label{def:shake}
    Let $L=L_1\cup\cdots\cup L_m$ and $L'=L'_1\cup\cdots\cup L'_m$ be ordered, oriented links, with fixed zero-framed tubular neighbourhoods. The links are \textit{shake concordant} if there are pairwise disjoint, smooth, properly embedded, compact, connected, oriented genus-zero surfaces
    \[
        F_1,\ldots,F_m\subset S^3\times[0,1]
    \]
    with $\partial F_k=(-Q_k)\sqcup Q'_k$, where $Q_k\subset S^3\times\{0\}$ and $Q'_k\subset S^3\times\{1\}$. For some nonnegative integers $a_{kj}$, the link $Q_k$ consists of $2a_{kk}+1$ zero-framed parallel copies of $L_k$, of which $a_{kk}+1$ have its orientation, together with $a_{kj}$ oppositely oriented pairs of zero-framed parallel copies of $L_j$ for every $j\neq k$. The link $Q'_k$ satisfies the same condition with respect to $L'$ and independently chosen integers $b_{kj}\geq 0$.
    
    All copies of a given component, including copies assigned to different surfaces, are distinct parallel curves in its fixed framed neighbourhood. The shake concordance is \textit{strong} if $a_{kj}=b_{kj}=0$ for all $j\neq k$. We write $L\shsim L'$ if the links are joined by a finite chain of shake concordances.
\end{definition}

Recall that oriented closed $3$-manifolds $Y_0,Y_1$ are \emph{smoothly integral homology cobordant} if they cobound a compact smooth oriented $4$-manifold $W$, with $\partial W=(-Y_0)\sqcup Y_1$, such that both boundary inclusions induce homology isomorphisms. For an $m$-component link $L$, its zero-framed linking matrix $\Lambda$ has entries
\[
    \Lambda_{ii}=0,\qquad \Lambda_{ij}=\operatorname{lk}(L_i,L_j)\quad(i\neq j).
\]
Since $H_1(M_L)\cong\operatorname{coker}\Lambda$, unimodularity of $\Lambda$ is equivalent to $M_L$ being an integral homology sphere.

\begin{proposition}[cf. {\cite[Proposition 4.1]{Bos21}}]\label{prop:homcob}
    Suppose that $L$ and $L'$ are shake concordant and that the zero-framed linking matrix of $L$ is unimodular. Then the linking matrix of $L'$ is also unimodular, and $M_L$ and $M_{L'}$ are smoothly integral homology cobordant.
\end{proposition}

\begin{proof}
We cap the shake surfaces to obtain disjoint square-zero spheres, perform surgery on these spheres, and use unimodularity to verify that the resulting cobordism is an integral homology cobordism.

Attach zero-framed $2$-handles to $S^3\times[0,1]$ along $L$ and $L'$ at its two ends, obtaining a manifold $W$ with
\[
    \partial W=(-M_L)\sqcup M_{L'}.
\]
The handle calculation gives
\[
    H_1(W)=0,\qquad H_2(W)\cong\mathbb{Z}^{2m},\qquad
    H_3(W)\cong\mathbb{Z},\qquad H_4(W)=0.
\]
Orient the handle cores so that their boundaries agree with the given orientations of the corresponding link components. Capping these cores with suitably oriented Seifert surfaces pushed into the end collars gives a basis $e_1,\ldots,e_m,e'_1,\ldots,e'_m$ for $H_2(W)$. If $\Lambda$ and $\Lambda'$ are the linking matrices of $L$ and $L'$, then the intersection form is
\begin{equation}\label{eq:intersection-form}
    Q_W=(-\Lambda)\oplus\Lambda'.
\end{equation}

Cap every boundary component of each $F_k$ with a distinct parallel copy of the corresponding handle core. After smoothing corners, this gives disjoint embedded spheres $\Sigma_1,\ldots,\Sigma_m$. The signed boundary counts give a contribution of $e_k$ from the bottom caps and $-e'_k$ from the top caps, so
\[
    [\Sigma_k]=s_k=e_k-e'_k.
\]
Disjointness gives $s_k\cdot s_l=0$ for $k\neq l$, while the zero framings give
\[
    s_k^2=-\Lambda_{kk}+\Lambda'_{kk}=0.
\]
Thus the spheres have trivial normal bundles, and their pairings give
\begin{equation}\label{eq:linking-matrices}
    \Lambda=\Lambda'.
\end{equation}
In particular, $\Lambda'$ is unimodular. The subgroup $S=\langle s_1,\ldots,s_m\rangle$ is an isotropic direct summand of $H_2(W)$, since $s_1,\ldots,s_m,e'_1,\ldots,e'_m$ form an integral basis.

Let $X=W\setminus\operatorname{int}\nu(\Sigma_1\cup\cdots\cup\Sigma_m)$. It is connected by general position. Excision identifies $H_j(W,X)$ with $\mathbb{Z}^m$ for $j=2,4$ and with zero otherwise. The map in degree two is
\begin{equation}\label{eq:intersection-map}
    H_2(W)\longrightarrow H_2(W,X)\cong\mathbb{Z}^m,
    \qquad x\longmapsto(x\cdot s_k)_{k=1}^m.
\end{equation}
Its restriction to $\langle e_1,\ldots,e_m\rangle$ is represented by $-\Lambda$, and hence is surjective. The exact sequence of $(W,X)$ gives
\begin{equation}\label{eq:complement}
    H_1(X)=0,\qquad H_2(X)\cong S^\perp,
\end{equation}
as well as a short exact sequence
\begin{equation}\label{eq:tube-classes}
    0\longrightarrow\mathbb{Z}^m\longrightarrow H_3(X)
    \longrightarrow H_3(W)\cong\mathbb{Z}\longrightarrow0.
\end{equation}
The first map in \eqref{eq:tube-classes} sends the standard generators to the classes $[\partial\nu\Sigma_k]$.

The kernel $S^\perp$ of \eqref{eq:intersection-map} has rank $m$ and contains $S$. Since $S$ is a direct summand, $S^\perp/S$ is a torsion-free subgroup of $H_2(W)/S$. It has rank zero, so
\begin{equation}\label{eq:lagrangian}
    S^\perp=S.
\end{equation}
Perform surgery on the spheres by replacing their neighbourhoods $S^2\times D^2$ with $D^3\times S^1$, and denote the resulting manifold by $W'$. The nonzero relative groups $H_j(W',X)$ are $\mathbb{Z}^m$ in degrees three and four. In degree three, the boundary map sends each generator to a parallel copy of $\Sigma_k$. Under \eqref{eq:complement}, these classes correspond to $s_k$, so \eqref{eq:lagrangian} makes this boundary map an isomorphism onto $H_2(X)$. In degree four, the boundary map has image $\langle[\partial\nu\Sigma_k]\rangle$ and is injective by \eqref{eq:tube-classes}. It follows that
\[
    H_1(W')=H_2(W')=H_4(W')=0,\qquad H_3(W')\cong\mathbb{Z}.
\]

The manifold $W'$ is connected, its boundary is still $(-M_L)\sqcup M_{L'}$, and both boundary components are integral homology spheres. Poincar\'e--Lefschetz duality gives $H_3(W',\partial W')\cong H^1(W')=0$. Hence $H_3(\partial W')\cong\mathbb{Z}^2\to H_3(W')\cong\mathbb{Z}$ is surjective, with kernel generated by the boundary of the relative fundamental class. This generator has coefficient $\pm1$ on each boundary component, so each component maps isomorphically onto $H_3(W')$. Connectedness and the remaining homology groups show that both boundary inclusions are homology isomorphisms.
\end{proof}

\begin{corollary}\label{cor:chain}
    On links with unimodular zero-framed linking matrix, the assignment
    \[
        L\longmapsto[M_L]\in\ThZ
    \]
    is constant on shake-equivalence classes. It is also unchanged by reordering or reorienting components.
\end{corollary}

\begin{proof}
By \eqref{eq:linking-matrices}, unimodularity is preserved along finite chains of shake concordances. Applying Proposition \ref{prop:homcob} to consecutive links and composing the resulting homology cobordisms proves the first assertion. Reordering or reorienting a surgery link leaves its oriented surgery manifold unchanged.
\end{proof}

For locally knotted Hopf links, the matrix is
\[
    \begin{pmatrix}0&\varepsilon\\\varepsilon&0\end{pmatrix},
    \qquad \varepsilon\in\{1,-1\},
\]
so Corollary \ref{cor:chain} applies to every link considered in Theorem \ref{thm:main}.

%----------------------------------------------------%
%----------------------------------------------------%
\section{Whitehead-link surgery identities}\label{sec:whitehead}

%----------------------------------------------------%
\subsection{Splices and satellite exteriors}
For an oriented knot $K$, write $E(K)=S^3\setminus\operatorname{int}\nu(K)$, and let $\mu_K,\lambda_K$ be its positive meridian and preferred Seifert longitude. Our surgery convention is
\[
    S^3_{p/q}(K)=E(K)(p\mu_K+q\lambda_K),
\]
where $p,q$ are coprime integers and the right-hand side is Dehn filling along the indicated unoriented slope. If $C$ is a component of a link with exterior $Y$, write $Y(C;\alpha)$ for filling the $C$-boundary along $\alpha$.

A knot $K$ is \textit{reversible} if an orientation-preserving diffeomorphism of $S^3$ carries it to itself with the opposite orientation. Such a diffeomorphism induces $-\operatorname{id}$ on $H_1(\partial E(K))$. For knot exteriors $E(J)$ and $E(K)$, the \textit{standard splice gluing} is the orientation-reversing boundary map specified, up to isotopy, by
\[
    \mu_J\longmapsto\lambda_K,\qquad \lambda_J\longmapsto\mu_K.
\]
If either knot is reversible, reversing both signs does not change the oriented diffeomorphism type of the resulting manifold.

% The Hopf-link surgery description
\begin{lemma}\label{lem:splice}
    The manifold $M_{\mathcal{H}(J,K)}$ is obtained by gluing $E(J)$ to $E(K)$ by a map of the form
    \begin{equation}\label{eq:splice-map}
        \mu_J\longmapsto\varepsilon\lambda_K,
        \qquad \lambda_J\longmapsto\varepsilon\mu_K,
        \qquad \varepsilon\in\{1,-1\}.
    \end{equation}
    If either knot is reversible, the standard splice gluing may be used. In particular,
    \[
        M_{\mathcal{H}(U,K)}\cong S^3
    \]
    for every knot $K$.
\end{lemma}

\begin{proof}
The zero-surgery description of a meridian-longitude splice is the Hopf link with the two knots tied into its components; see \cite[proof of Lemma 2.1 and Figure 1]{KLT}. The gluing exchanges the meridian and longitude slopes, and the two signs agree because it reverses the boundary orientation. This gives \eqref{eq:splice-map}. If either knot is reversible, an orientation-preserving diffeomorphism of its exterior negates both peripheral classes. Composing the gluing with this diffeomorphism changes both signs, so the standard splice gluing may be used. Finally, $E(U)$ is a solid torus with meridian $\lambda_U$, which is identified with $\varepsilon\mu_K$. Since filling slopes are unoriented, the resulting manifold is the meridional filling of $E(K)$, namely $S^3$.
\end{proof}

Let $\mathcal{W}_+=P\cup A$ be the Whitehead link whose component $P$ is the positive-clasped, untwisted pattern in the solid torus
\[
    V=S^3\setminus\operatorname{int}\nu(A).
\]
Both components are unknotted and $\operatorname{lk}(P,A)=0$. A meridian and longitude for $V$ are $\lambda_A$ and $\mu_A$, respectively. For $t\in\mathbb{Z}$, an orientation-preserving diffeomorphism $h_t\colon V\to\nu(J)$ with
\begin{equation}\label{eq:pattern}
    h_t(\lambda_A)=\mu_J,
    \qquad h_t(\mu_A)=\lambda_J+t\mu_J
\end{equation}
defines $D_+(J,t)=h_t(P)$. This agrees with the convention in \cite[Figure 1]{Hed}. We normalize the Alexander polynomial by $\Delta_K(x)=\Delta_K(x^{-1})$ and $\Delta_K(1)=1$. The standard genus-one Seifert surface and its Seifert matrix give
\begin{equation}\label{eq:whitehead-basic}
    g_4(D_+(J,t))\leq 1,
    \qquad \Delta_{D_+(J,t)}(x)=-tx+(2t+1)-tx^{-1}.
\end{equation}
Here $g_4$ is the smooth slice genus. In particular, $\Delta_{\mathcal{D}(J)}=1$ for every $J$. The link and pattern conventions are illustrated in Figure \ref{fig:conventions}.

\begin{figure}[htbp]
\centering
    \begin{tikzpicture}[font=\small,scale=1.00,link strand/.style={draw=black,line width=.8pt,line cap=round,line join=round},crossing/.style={link strand,preaction={draw=white,line width=3.6pt,line cap=butt}}]
        \begin{scope}
        % Left clockwise, right counterclockwise.
        \draw[link strand,postaction={decorate},decoration={markings,mark=at position .62 with {\arrow{stealth}}}] (.2,-.72) arc[start angle=0,end angle=-360,x radius=.7,y radius=1.1];
        \draw[link strand,postaction={decorate},decoration={markings,mark=at position .11 with {\arrow{stealth}}}] (1.2,-.72) arc[start angle=0,end angle=360,x radius=.7,y radius=1.1];
        % Right component over at the top, left over at the bottom.
        \draw[crossing] (.15,.232628) arc[start angle=120,end angle=152,x radius=.7,y radius=1.1];
        \draw[crossing] (.118063,-1.236419) arc[start angle=332,end angle=298,x radius=.7,y radius=1.1];
        \filldraw[fill=white,line width=.5pt] (-1.48,-.96) rectangle (-.92,-.48);
        \node at (-1.2,-.72) {$J$};
        \filldraw[fill=white,line width=.5pt] (.92,-.96) rectangle (1.48,-.48);
        \node at (1.2,-.72) {$K$};
        \node at (0,-2.17) {$\mathcal H(J,K)$};
        \end{scope}
        \begin{scope}[shift={(6,-.43)}]
        \draw[link strand] (-1.25,.02) ellipse[x radius=.82,y radius=.4];
        % Outer arc of P
        \draw[crossing,postaction={decorate},decoration={markings,mark=at position .77 with {\arrow{stealth}}}]
        (-.55,1.80) .. controls (-1.04,2.01) and (-1.60,1.10) .. (-1.60,.10)
        .. controls (-1.60,-.98) and (-1.10,-1.45) .. (0,-1.45)
        .. controls (1.10,-1.45) and (1.60,-.98) .. (1.60,.10)
        .. controls (1.60,1.10) and (1.04,2.01) .. (.55,1.80);
        % Inner arc of P
        \draw[crossing,postaction={decorate},decoration={markings,mark=at position .875 with {\arrow{stealth}}}]
        (.55,.30) .. controls (.97,.12) and (.97,-.70) .. (0,-.70)
        .. controls (-.97,-.70) and (-.97,.12) .. (-.55,.30);
        % The two clasp strands
        \draw[link strand] (-.55,1.80) .. controls (-.20,1.65) and (.32,1.30) .. (.32,1.05) .. controls (.32,.80) and (-.20,.45) .. (-.55,.30);
        \draw[link strand] (.55,1.80) .. controls (.20,1.65) and (-.32,1.30) .. (-.32,1.05) .. controls (-.32,.80) and (.20,.45) .. (.55,.30);
        % Hedden's positive clasp
        \draw[crossing] (-0.144388,1.578415) .. controls (-0.044596,1.512347) and (0.050150,1.439950) .. (0.127725,1.366495);
        \draw[crossing] (-0.127725,0.733505) .. controls (-0.050150,0.660050) and (0.044596,0.587653) .. (0.144388,0.521585);
        % The upper arc of A passes over both strands of P and is directed to the right.
        \draw[crossing,postaction={decorate},decoration={markings,mark=at position .5 with {\arrow{stealth}}}] (-2.07,.02) arc[start angle=180,end angle=0,x radius=.82,y radius=.4];
        \node at (-2.10,-.51) {$A$};
        \node at (1.83,.10) {$P$};
        \node[font=\scriptsize] at (0,2.25) {positive clasp};
        \draw[->,line width=.4pt] (0,2.03)--(0,1.62);
        \node at (0,-1.74) {$\mathcal W_+=P\cup A$};
        \node[font=\scriptsize] at (0,-2.27) {$(\mu_V,\lambda_V)=(\lambda_A,\mu_A)$};
        \end{scope}
    \end{tikzpicture}
    \caption{Local knotting and Whitehead-pattern conventions. In the left panel, each box represents an oriented one-strand local knot tied in a disjoint ball; the two crossings outside the boxes form the positive Hopf clasp. In the right panel, $P$ is the positive untwisted Whitehead pattern in $V=S^3\setminus\operatorname{int}\nu(A)$. The peripheral convention is specified by \eqref{eq:pattern}; the diagram uses the positive clasp convention of \cite{Hed}.}
    \label{fig:conventions}
\end{figure}
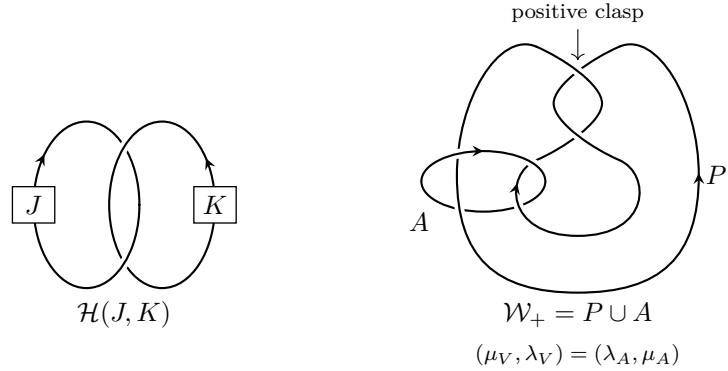

\begin{lemma}\label{lem:satellite}
    For every knot $J$,
    \[
        E(\mathcal{D}(J))\cong E(J)\cup_A E(\mathcal{W}_+),
    \]
    where the gluing identifies $\mu_J$ with $\lambda_A$ and $\lambda_J$ with $\mu_A$. On the remaining boundary, $\mu_P$ and $\lambda_P$ correspond to the meridian and preferred longitude of $\mathcal{D}(J)$, respectively.
\end{lemma}

\begin{proof}
The exterior decomposition and gluing map follow from \eqref{eq:pattern} with $t=0$. Since $\operatorname{lk}(P,A)=0$, a Seifert surface for $P$ can be tubed along pairs of oppositely signed intersections with $A$, keeping its boundary collar fixed, to obtain a Seifert surface $F\subset V$. The image $h_0(F)$ is a Seifert surface for $\mathcal{D}(J)$, so $h_0$ sends $\lambda_P$ to its preferred longitude. The meridian identification follows directly from the orientation-preserving embedding $h_0$.
\end{proof}

%----------------------------------------------------%
\subsection{Symmetries and peripheral signs}
We use two types of symmetry of the Whitehead link: exchange of its components, and reversal of either component while the other is preserved. We record the associated peripheral actions (i.e., their actions on the meridian-longitude pairs).

\Needspace{15\baselineskip}
\begin{lemma}\label{lem:symmetries}
    There are orientation-preserving diffeomorphisms of $(S^3,\mathcal{W}_+)$ with the following properties:
    \begin{enumerate}
        \item $\sigma$ exchanges $P$ and $A$, preserving their orientations and their meridian--longitude pairs;
        \item $\rho_P$ fixes both components setwise, reverses the orientation of $P$, and preserves that of $A$; its peripheral actions are $-\operatorname{id}$ at $P$ and $\operatorname{id}$ at $A$;
        \item $\rho_A$ fixes both components setwise, preserves the orientation of $P$, and reverses that of $A$; its peripheral actions are $\operatorname{id}$ at $P$ and $-\operatorname{id}$ at $A$.
    \end{enumerate}
    Here the peripheral actions are the induced maps on the first homology of the boundary tori.
\end{lemma}

\begin{proof}
The pure exchange and a single-component reversal are realized by explicit isotopies in \cite[Claim 7.14 and Figures B5 and C16]{BCC}. Conjugating the reversal by the exchange gives the reversal of the other component. Conjugation by a reflection gives the same symmetries for the other chirality of the Whitehead link. After isotopy near the link, we may assume that these symmetries carry the chosen tubular neighbourhoods to the corresponding tubular neighbourhoods. An orientation-preserving ambient symmetry preserves both oriented peripheral classes when it preserves the component orientation, and negates both when it reverses that orientation. This determines the stated actions.
\end{proof}

Peripheral identifications will be used up to isotopy of the boundary tori. An orientation-preserving torus diffeomorphism inducing the identity on first homology is isotopic to the identity. Thus, whenever necessary, we adjust the maps in boundary collars before gluing; these adjustments do not change the stated peripheral actions.

\begin{lemma}\label{lem:reversibility}
    For every knot $J$ and every $t\in\mathbb{Z}$, the Whitehead double $D_+(J,t)$ is reversible. In particular, $\mathcal{D}(J)$ and all the knots $D_+(U,-n)$ are reversible.
\end{lemma}

\begin{proof}
Restrict $\rho_P$ to $V$. This map reverses the oriented pattern $P$ and induces the identity on $H_1(\partial V)$. Its boundary restriction is therefore isotopic to the identity. Modify the map in a collar of $\partial V$, disjoint from $P$, so that it is the identity near the boundary. Conjugating by $h_t$ gives an orientation-preserving diffeomorphism of $\nu(J)$ that reverses $D_+(J,t)$ and is the identity near $\partial\nu(J)$. Extending by the identity on $E(J)$ gives an orientation-preserving diffeomorphism of $S^3$ reversing $D_+(J,t)$.
\end{proof}

%----------------------------------------------------%
\subsection{Twist-knot fillings}
The following filling calculation fixes the sign in \eqref{eq:twist}.

\begin{lemma}\label{lem:filling}
    For every nonzero integer $n$, there is an orientation-preserving diffeomorphism
    \begin{equation}\label{eq:twist-exterior}
        E(D_+(U,-n))\cong E(\mathcal{W}_+)(P;\mu_P+n\lambda_P).
    \end{equation}
    Under this identification, the meridian--longitude pair on the remaining $A$-boundary corresponds to the meridian and preferred longitude of $D_+(U,-n)$, up to simultaneous sign.
\end{lemma}

\begin{proof}
We first fill the $A$-boundary. Since $\lambda_A$ is a meridian of $V=S^3\setminus\operatorname{int}\nu(A)$, an appropriate power of a meridional disk twist gives an orientation-preserving diffeomorphism $g_n\colon V\to V$ with peripheral action
\[
    g_n(\lambda_A)=\lambda_A,
    \qquad g_n(\mu_A)=\mu_A-n\lambda_A.
\]
Let $h_0\colon V\to\nu(U)$ be the zero-framed identification from \eqref{eq:pattern}, and set $f_n=h_0\circ g_n$. Then
\[
    f_n(\lambda_A)=\mu_U,
    \qquad f_n(\mu_A)=\lambda_U-n\mu_U,
\]
and consequently
\[
    f_n(\mu_A+n\lambda_A)=\lambda_U.
\]
The curve $\lambda_U$ is a meridian of the complementary solid torus $E(U)$. Thus $f_n$ extends over the filling solid torus to an orientation-preserving diffeomorphism
\[
    S^3_{1/n}(A)\longrightarrow \nu(U)\cup E(U)=S^3.
\]
By \eqref{eq:pattern}, the image of $P$ is $D_+(U,-n)$. Restricting to its exterior gives
\[
    E(\mathcal{W}_+)(A;\mu_A+n\lambda_A)
    \cong E(D_+(U,-n));
\]
see Figure \ref{fig:twist-filling}.

The symmetry $\sigma$ exchanges $P$ and $A$ and their peripheral pairs. It therefore identifies the filling on the left with $E(\mathcal{W}_+)(P;\mu_P+n\lambda_P)$, proving \eqref{eq:twist-exterior}. A meridian of the remaining component is preserved. Since $\operatorname{lk}(A,P)=0$, a Seifert surface for $A$ can be chosen disjoint from $P$ and survives the filling on $P$. Its boundary collar determines the original preferred longitude $\lambda_A$, so this longitude remains the preferred longitude of the image knot. Reversing the orientation of that knot changes the signs of both peripheral classes simultaneously.
\end{proof}

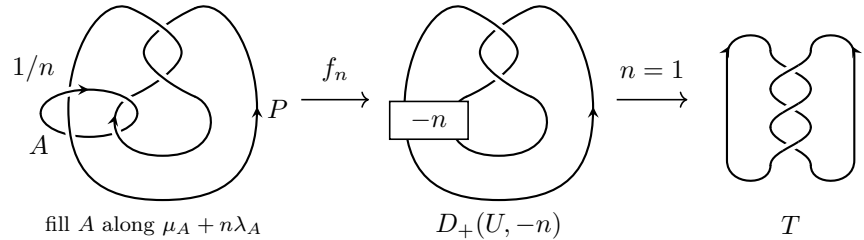
\begin{figure}[htbp]
\centering
    \begin{tikzpicture}[font=\small,scale=.78,
    link strand/.style={draw=black,line width=.8pt,line cap=round,line join=round},
    crossing/.style={link strand,preaction={draw=white,line width=3.6pt,line cap=butt}}]
    \begin{scope}
        \draw[link strand] (-1.25,.02) ellipse[x radius=.82,y radius=.4];
        % Outer arc of P
        \draw[crossing,postaction={decorate},decoration={markings,mark=at position .77 with {\arrow{stealth}}}]
        (-.55,1.80) .. controls (-1.04,2.01) and (-1.60,1.10) .. (-1.60,.10)
        .. controls (-1.60,-.98) and (-1.10,-1.45) .. (0,-1.45)
        .. controls (1.10,-1.45) and (1.60,-.98) .. (1.60,.10)
        .. controls (1.60,1.10) and (1.04,2.01) .. (.55,1.80);
        % Inner arc of P
        \draw[crossing,postaction={decorate},decoration={markings,mark=at position .875 with {\arrow{stealth}}}]
        (.55,.30) .. controls (.97,.12) and (.97,-.70) .. (0,-.70)
        .. controls (-.97,-.70) and (-.97,.12) .. (-.55,.30);
        % The two clasp strands
        \draw[link strand] (-.55,1.80) .. controls (-.20,1.65) and (.32,1.30) .. (.32,1.05) .. controls (.32,.80) and (-.20,.45) .. (-.55,.30);
        \draw[link strand] (.55,1.80) .. controls (.20,1.65) and (-.32,1.30) .. (-.32,1.05) .. controls (-.32,.80) and (.20,.45) .. (.55,.30);
        % Hedden's positive clasp
        \draw[crossing] (-0.144388,1.578415) .. controls (-0.044596,1.512347) and (0.050150,1.439950) .. (0.127725,1.366495);
        \draw[crossing] (-0.127725,0.733505) .. controls (-0.050150,0.660050) and (0.044596,0.587653) .. (0.144388,0.521585);
        % The upper arc of A passes over both strands of P and is directed to the right.
        \draw[crossing,postaction={decorate},decoration={markings,mark=at position .5 with {\arrow{stealth}}}]
        (-2.07,.02) arc[start angle=180,end angle=0,x radius=.82,y radius=.4];
        \node at (-2.10,-.51) {$A$};
        \node at (1.95,.10) {$P$};
        \node at (-2.19,.78) {$1/n$};
        \node[font=\scriptsize] at (-.15,-1.88) {fill $A$ along $\mu_A+n\lambda_A$};
        \end{scope}
        \draw[->,line width=.6pt] (2.35,.25)--(3.50,.25);
        \node at (2.925,.76) {$f_n$};
        \begin{scope}[shift={(5.7,0)}]
        % Outer arc of P
        \draw[crossing,postaction={decorate},decoration={markings,mark=at position .77 with {\arrow{stealth}}}]
        (-.55,1.80) .. controls (-1.04,2.01) and (-1.60,1.10) .. (-1.60,.10)
        .. controls (-1.60,-.98) and (-1.10,-1.45) .. (0,-1.45)
        .. controls (1.10,-1.45) and (1.60,-.98) .. (1.60,.10)
        .. controls (1.60,1.10) and (1.04,2.01) .. (.55,1.80);
        % Inner arc of P
        \draw[crossing,postaction={decorate},decoration={markings,mark=at position .875 with {\arrow{stealth}}}]
        (.55,.30) .. controls (.97,.12) and (.97,-.70) .. (0,-.70)
        .. controls (-.97,-.70) and (-.97,.12) .. (-.55,.30);
        % The two clasp strands
        \draw[link strand] (-.55,1.80) .. controls (-.20,1.65) and (.32,1.30) .. (.32,1.05) .. controls (.32,.80) and (-.20,.45) .. (-.55,.30);
        \draw[link strand] (.55,1.80) .. controls (.20,1.65) and (-.32,1.30) .. (-.32,1.05) .. controls (-.32,.80) and (.20,.45) .. (.55,.30);
        % Hedden's positive clasp
        \draw[crossing] (-0.144388,1.578415) .. controls (-0.044596,1.512347) and (0.050150,1.439950) .. (0.127725,1.366495);
        \draw[crossing] (-0.127725,0.733505) .. controls (-0.050150,0.660050) and (0.044596,0.587653) .. (0.144388,0.521585);
        \filldraw[fill=white,line width=.6pt] (-1.85,-.42) rectangle (-.52,.22);
        \node at (-1.185,-.10) {$-n$};
        \node at (0,-1.88) {$D_+(U,-n)$};
        \end{scope}
        \draw[->,line width=.6pt] (7.70,.25)--(8.90,.25);
        \node at (8.30,.76) {$n=1$};
        \begin{scope}[shift={(10.65,1.08)},scale=.90]
        % The three crossings are positive with both braid strands directed downward.
        \pgfmathsetmacro{\braidheight}{.72*3}
        \foreach \k in {1,...,3}{%
          \pgfmathsetmacro{\ytop}{-.72*(\k-1)}
          \pgfmathsetmacro{\ybottom}{-.72*\k}
          \draw[link strand] (.38,\ytop)
            .. controls (.38,{\ytop-.30}) and (-.38,{\ybottom+.30}) .. (-.38,\ybottom);
          \draw[link strand] (-.38,\ytop)
            .. controls (-.38,{\ytop-.30}) and (.38,{\ybottom+.30}) .. (.38,\ybottom);
          \begin{scope}
            \clip (0,{\ytop-.36}) circle[radius=.17];
            \draw[crossing] (.38,\ytop)
              .. controls (.38,{\ytop-.30}) and (-.38,{\ybottom+.30}) .. (-.38,\ybottom);
          \end{scope}
        }
        \draw[link strand,postaction={decorate},decoration={markings,mark=at position .77 with {\arrow{stealth}}}]
        (-.38,{-\braidheight})
        .. controls (-.38,{-\braidheight-.4}) and (-1.2,{-\braidheight-.4}) .. (-1.2,{-\braidheight+.1})
        -- (-1.2,-.1)
        .. controls (-1.2,.4) and (-.38,.4) .. (-.38,0);
        \draw[link strand,postaction={decorate},decoration={markings,mark=at position .77 with {\arrow{stealth}}}]
        (.38,{-\braidheight})
        .. controls (.38,{-\braidheight-.4}) and (1.2,{-\braidheight-.4}) .. (1.2,{-\braidheight+.1})
        -- (1.2,-.1)
        .. controls (1.2,.4) and (.38,.4) .. (.38,0);
        \node at (0,-3.28) {$T$};
    \end{scope}
    \end{tikzpicture}
    \caption{The twist-filling convention. After $1/n$-surgery on $A$, the extension of $f_n=h_0\circ g_n$ carries $P$ to $D_+(U,-n)$ in $S^3$. The box records Hedden's full-twist parameter $t=-n$, rather than a signed crossing count. For $n=1$, the resulting knot is the right-handed trefoil $T$. Exchanging $P$ and $A$ gives the filling used in Lemma \ref{lem:filling}.}
    \label{fig:twist-filling}
\end{figure}

\begin{proof}[Proof of Theorem \ref{thm:surgery}, identity \eqref{eq:twist}]
By Lemmas \ref{lem:splice} and \ref{lem:reversibility}, $M_{\mathcal{H}(K,D_+(U,-n))}$ is the standard splice of $E(K)$ and $E(D_+(U,-n))$. Substitute the exterior description from Lemma \ref{lem:filling}. If the peripheral identification introduces a simultaneous minus sign on the $A$-boundary, compose with $\rho_A$. This symmetry negates both peripheral classes at $A$ and fixes those at $P$, so it extends across the $P$-filling and removes the sign. We may therefore use the gluing
\[
    \mu_K\longleftrightarrow\lambda_A,
    \qquad \lambda_K\longleftrightarrow\mu_A.
\]
It follows that
\begin{align*}
    M_{\mathcal{H}(K,D_+(U,-n))}
    &\cong\bigl(E(K)\cup_A E(\mathcal{W}_+)\bigr)
    (P;\mu_P+n\lambda_P)\\
    &\cong E(\mathcal{D}(K))
    \bigl(\mu_{\mathcal{D}(K)}+n\lambda_{\mathcal{D}(K)}\bigr)\\
    &=S^3_{1/n}(\mathcal{D}(K)).
\end{align*}
The second identification, including the preferred longitude, is Lemma \ref{lem:satellite}.
\end{proof}

For $n=1$, the knot $D_+(U,-1)$ is the right-handed trefoil, so \eqref{eq:twist} recovers Tange's trefoil formula. For $n=-1$, $D_+(U,1)$ is the figure-eight knot, giving the corresponding $-1$-surgery formula; see \cite[Section 1.1]{Tan} and the special cases in \cite[Theorem 1.1]{Tsu}.

%----------------------------------------------------%
\subsection{Transferring the Whitehead pattern}

\begin{proof}[Proof of Theorem \ref{thm:surgery}, identity \eqref{eq:transfer}]
Since $\mathcal{D}(J)$ is reversible, Lemma \ref{lem:splice} gives the standard splice description for $M_{\mathcal{H}(\mathcal{D}(J),K)}$. Expanding the double by Lemma \ref{lem:satellite} yields
\begin{equation}\label{eq:three-pieces}
    M_{\mathcal{H}(\mathcal{D}(J),K)}\cong E(J)\cup_AE(\mathcal{W}_+)\cup_PE(K),
\end{equation}
with meridian--longitude gluings at both boundaries of $E(\mathcal{W}_+)$; see Figure \ref{fig:transfer}.

By Lemma \ref{lem:symmetries}, the symmetry $\sigma$ exchanges these boundaries and preserves their preferred peripheral pairs. It is therefore compatible with the gluings up to isotopy. Applying $\sigma$ to the middle piece and the identity maps to $E(J)$ and $E(K)$, with the boundary-collar adjustments described above, gives an orientation-preserving diffeomorphism
\[
    E(J)\cup_AE(\mathcal{W}_+)\cup_PE(K)
    \cong E(K)\cup_AE(\mathcal{W}_+)\cup_PE(J).
\]
The first two pieces on the right form $E(\mathcal{D}(K))$, with its preferred peripheral pair on the remaining boundary, by Lemma \ref{lem:satellite}. Thus the resulting manifold is the standard splice of $E(J)$ and $E(\mathcal{D}(K))$. Since $\mathcal{D}(K)$ is reversible, Lemma \ref{lem:splice} identifies this splice with $M_{\mathcal{H}(J,\mathcal{D}(K))}$.
\end{proof}

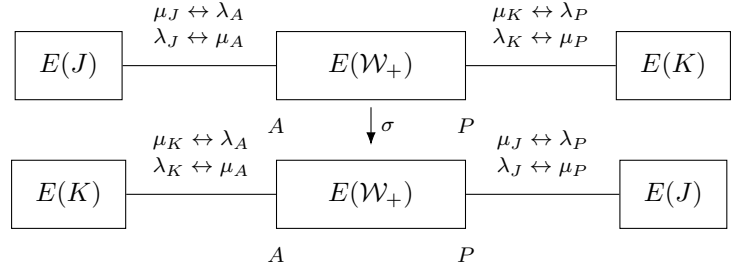
\begin{figure}[htbp]
\centering
    \begin{tikzpicture}[font=\small,>=Latex,
     piece/.style={draw,minimum height=9mm,inner xsep=9pt},
     gluing/.style={font=\scriptsize,align=center}]
    \node[piece] (j) at (0,0) {$E(J)$};
    \node[piece,minimum width=25mm] (w) at (4,0) {$E(\mathcal{W}_+)$};
    \node[piece] (k) at (8,0) {$E(K)$};
    \draw (j.east)--node[gluing,above=3pt] {$\mu_J\leftrightarrow\lambda_A$\\$\lambda_J\leftrightarrow\mu_A$} (w.west);
    \draw (w.east)--node[gluing,above=3pt] {$\mu_K\leftrightarrow\lambda_P$\\$\lambda_K\leftrightarrow\mu_P$} (k.west);
    \node[font=\scriptsize,below=3pt] at (w.south west) {$A$};
    \node[font=\scriptsize,below=3pt] at (w.south east) {$P$};
    \node[piece] (k2) at (0,-1.7) {$E(K)$};
    \node[piece,minimum width=25mm] (w2) at (4,-1.7) {$E(\mathcal{W}_+)$};
    \node[piece] (j2) at (8,-1.7) {$E(J)$};
    \draw (k2.east)--node[gluing,above=3pt] {$\mu_K\leftrightarrow\lambda_A$\\$\lambda_K\leftrightarrow\mu_A$} (w2.west);
    \draw (w2.east)--node[gluing,above=3pt] {$\mu_J\leftrightarrow\lambda_P$\\$\lambda_J\leftrightarrow\mu_P$} (j2.west);
    \node[font=\scriptsize,below=3pt] at (w2.south west) {$A$};
    \node[font=\scriptsize,below=3pt] at (w2.south east) {$P$};
    \draw[->] (4,-0.55)--node[right,font=\scriptsize] {$\sigma$} (4,-1.05);
    \end{tikzpicture}
    \caption{The transfer identity. The symmetry $\sigma$ exchanges the Whitehead boundary tori, moving each companion exterior to the other boundary while preserving the specified meridian--longitude identifications.}
    \label{fig:transfer}
\end{figure}

%----------------------------------------------------%
%----------------------------------------------------%
\section{Floer obstructions and the proof of the main theorem}\label{sec:floer}

%----------------------------------------------------%
\subsection{Correction terms}
We use Heegaard Floer invariants with coefficients in $\mathbb{F}_2=\mathbb{Z}/2\mathbb{Z}$. Hedden's formula \cite[Theorem 1.5]{Hed} gives
\begin{equation}\label{eq:hedden}
    \tau(D_+(K,t))=
    \begin{cases}
        1,&t<2\tau(K),\\
        0,&t\geq 2\tau(K).
    \end{cases}
\end{equation}
Since $\tau(T)=1$ and $\mathcal{D}(K)=D_+(K,0)$, induction gives $\tau(\mathcal{D}^j(T))=1$ for every $j\geq 1$. Together with the slice-genus bound $|\tau(K)|\leq g_4(K)$ and the genus-one upper bound in \eqref{eq:whitehead-basic}, this yields
\begin{equation}\label{eq:iterates}
    \tau(\mathcal{D}^j(T))=g_4(\mathcal{D}^j(T))=1\qquad(j\geq 1).
\end{equation}

For an integral homology sphere $Y$, let $d(Y)$ denote the Ozsv\'ath--Szab\'o correction term associated to its unique $\operatorname{Spin}^c$ structure. This invariant is preserved under smooth integral homology cobordism and satisfies $d(S^3)=0$ \cite[Theorem 1.2]{OS}.

For $k\geq 0$, the knot Floer invariants $V_k(K)$ are nonnegative integers satisfying
\[
    V_{k+1}(K)\leq V_k(K)\leq V_{k+1}(K)+1
\]
by \cite[Equation (2.3)]{HW}. The invariant
\[
    \nu^+(K)=\min\{k\geq 0:V_k(K)=0\}
\]
satisfies $\tau(K)\leq\nu^+(K)\leq g_4(K)$ \cite[Definition 2.1 and Propositions 2.3--2.4]{HW}.

For every positive integer $n$, the Ni--Wu surgery formula \cite[Proposition 1.6 and Remark 2.10]{NW} gives
\begin{equation}\label{eq:niwu}
    d(S^3_{1/n}(K))=-2V_0(K).
\end{equation}
To obtain \eqref{eq:niwu}, set $p=1$, $q=n$, and $i=0$ in the general formula and use $d(L(1,n))=d(S^3)=0$ and $\max\{V_0(K),V_1(K)\}=V_0(K)$.

\begin{lemma}\label{lem:d}
    If $\tau(K)=1$ and $g_4(K)\leq 1$, then $V_0(K)=1$ and
    \[
        d(S^3_{1/n}(K))=-2\qquad(n\geq 1).
    \]
    In particular, this conclusion holds with $K=\mathcal{D}(J)$ whenever $\tau(J)>0$.
\end{lemma}

\begin{proof}
The inequalities $1=\tau(K)\leq\nu^+(K)\leq g_4(K)\leq  1$ imply $\nu^+(K)=1$. Thus $V_1(K)=0$ and $V_0(K)>0$. Since $V_0(K)\leq V_1(K)+1=1$ and $V_0(K)$ is an integer, we obtain $V_0(K)=1$. Equation \eqref{eq:niwu} proves the first assertion. If $\tau(J)>0$, then \eqref{eq:hedden} with $t=0$ gives $\tau(\mathcal{D}(J))=1$, while the genus-one Whitehead surface gives $g_4(\mathcal{D}(J))\leq 1$, so the first assertion applies.
\end{proof}

%----------------------------------------------------%
\subsection{Linear independence of surgery classes}
To distinguish the surgery manifolds from one another, we use the following result from instanton Floer theory.

\begin{theorem}[Baldwin--Sivek {\cite[Corollary 1.5]{BS}}]\label{thm:BS}
    Let $K\subset S^3$ be a quasipositive knot that is not smoothly slice. Then the classes
    \[
        \bigl\{[S^3_{1/n}(K)]\bigr\}_{n\geq 1}
    \]
    are linearly independent in $\ThZ$.
\end{theorem}

Baldwin and Sivek show that the hypotheses above imply $h(S^3_1(K))<0$, where $h$ is the instanton Fr\o yshov invariant, and then apply the linear independence criterion of Nozaki, Sato and Taniguchi \cite[Theorem 1.8]{NST}.

\begin{lemma}\label{lem:Z}
    The knot $Z=\mathcal{D}^3(T)$ is strongly quasipositive and satisfies $\tau(Z)=g_4(Z)=1$.
\end{lemma}

\begin{proof}
The right-handed trefoil is strongly quasipositive, since it is the closure of the positive braid $\sigma_1^3$. By Rudolph's Whitehead-double construction, the positive untwisted double of a nontrivial strongly quasipositive knot is again strongly quasipositive \cite[Lemma 2]{Rud}. Each intermediate double is nontrivial by \eqref{eq:iterates}, so this construction can be applied three times to show that $Z$ is strongly quasipositive. Equation \eqref{eq:iterates} also gives $\tau(Z)=g_4(Z)=1$. In particular, $Z$ is quasipositive and not smoothly slice, so it satisfies the hypotheses of Theorem \ref{thm:BS}.
\end{proof}

%----------------------------------------------------%
\subsection{Proof of Theorem \ref{thm:main}}
We use the elementary fact that local knotting respects concordance in both categories.

\begin{lemma}\label{lem:local}
    If $J_i$ is concordant to $K_i$ for $i=1,2$, then $\mathcal{H}(J_1,J_2)$ is concordant to $\mathcal{H}(K_1,K_2)$. The statement holds for smooth concordance and for locally flat topological concordance.
\end{lemma}

\begin{proof}
Take the product concordance of the Hopf link and form its componentwise connected sum with concordances from $J_i$ to $K_i$, in disjoint local-knotting regions. The resulting surfaces are disjoint annuli from $\mathcal{H}(J_1,J_2)$ to $\mathcal{H}(K_1,K_2)$, respecting the component orderings and orientations. The construction preserves smoothness, respectively local flatness.
\end{proof}

\begin{proof}[Proof of Theorem \ref{thm:main}]
The local knots $\mathcal{D}(K_n)$ and $K_\ast=\mathcal{D}(T)$ have Alexander polynomial one by \eqref{eq:whitehead-basic}. Freedman's Alexander-polynomial-one theorem implies that both are locally flat topologically slice, and hence topologically concordant to the unknot $U$ \cite[Theorem 11.7B]{FQ}. Lemma \ref{lem:local} therefore gives a topological concordance from $L_n$ to $\mathcal{H}(U,U)=H$.

Apply the transfer identity \eqref{eq:transfer}, interchange the Hopf components, and then use the twist-knot surgery formula \eqref{eq:twist}:
\begin{align*}
    M_{L_n}
    &=M_{\mathcal{H}(\mathcal{D}(K_n),K_\ast)}\\
    &\cong M_{\mathcal{H}(K_n,\mathcal{D}^2(T))}\\
    &\cong M_{\mathcal{H}(\mathcal{D}^2(T),K_n)}\\
    &\cong S^3_{1/n}\bigl(\mathcal{D}^3(T)\bigr)
    =S^3_{1/n}(Z).
\end{align*}
All these diffeomorphisms preserve orientation, proving \eqref{eq:main-surgery}. By Lemma \ref{lem:Z}, the fixed knot $Z$ is quasipositive and not smoothly slice. Theorem \ref{thm:BS} therefore gives the linear independence of the classes $[M_{L_n}]$. Lemmas \ref{lem:Z} and \ref{lem:d} also give $d(M_{L_n})=-2$ for every $n\geq 1$.

The links $L_n$ and $H$ have linking number one, so their zero-framed linking matrices are unimodular. A finite chain of shake concordances from $L_n$ to $L_m$ would give $[M_{L_n}]=[M_{L_m}]$ by Corollary \ref{cor:chain}, contradicting linear independence when $n\neq m$. A finite chain from $L_n$ to $H$ would give $[M_{L_n}]=[M_H]=[S^3]=0$, again contradicting linear independence. Reordering or reorienting components leaves the oriented zero-surgery manifold unchanged, so the same obstructions apply when component orderings and orientations are disregarded.
\end{proof}

\begin{remark}\label{rem:elementary}
    The non-equivalence of $L_1=\mathcal{H}(K_\ast,K_\ast)$ and $H$ does not require Theorem \ref{thm:BS}. The surgery identities give
    \[
        M_{\mathcal{H}(K_\ast,K_\ast)}\cong S^3_1(Z),
    \]
    and Lemma \ref{lem:d} gives correction term $-2$, whereas $d(M_H)=0$. More generally, $d(M_{L_n})=-2$ separates every $L_n$ from $H$ by Corollary \ref{cor:chain} and homology-cobordism invariance of $d$. Since these correction terms are equal for all $n$, they do not distinguish the $L_n$ from one another. The instanton theorem supplies the pairwise distinction and the stronger linear-independence conclusion.
\end{remark}

\begin{remark}[Choice of companion]\label{rem:companion}
    The trefoil used to define the fixed companion $K_\ast=\mathcal{D}(T)$ can be replaced by any nontrivial strongly quasipositive knot $J$. More precisely, set
    \[
        L_n(J)=\mathcal{H}\bigl(\mathcal{D}(K_n),\mathcal{D}(J)\bigr),
    \qquad Z_J=\mathcal{D}^3(J).
    \]
    Then each $L_n(J)$ is topologically concordant to $H$, and there is an orientation-preserving diffeomorphism
    \[
        M_{L_n(J)}\cong S^3_{1/n}(Z_J).
    \]
    For each fixed $J$, the classes $\{[M_{L_n(J)}]\}_{n\geq 1}$ are linearly independent in $\ThZ$, and their correction terms are all $-2$. Indeed, strong quasipositivity and nontriviality give $\tau(J)=g(J)>0$ \cite[Proposition 3]{HW}, where $g$ denotes Seifert genus. Hedden's formula \eqref{eq:hedden}, the slice-genus bound $|\tau|\leq g_4$, and the genus-one upper bound in \eqref{eq:whitehead-basic} give $\tau(\mathcal{D}^j(J))=g_4(\mathcal{D}^j(J))=1$ for every $j\geq 1$. These iterates are nontrivial, so Rudolph's construction \cite[Lemma 2]{Rud} also shows inductively that they are strongly quasipositive. The proof of Theorem \ref{thm:main} therefore applies with $J$ in place of $T$.
\end{remark}

%----------------------------------------------------%
%----------------------------------------------------%
\section{Further examples, refinements, and questions}\label{sec:examples}

%----------------------------------------------------%
\subsection{Further examples}
The correction-term calculation also gives examples outside the topological concordance class of the Hopf link. Each example below has a trefoil component, whose nonzero classical signature obstructs topological sliceness.

\begin{proposition}\label{prop:examples}
    \begin{enumerate}
        \item If $\tau(K)>0$, then $\mathcal{H}(K,T)\not\shsim H$. In particular, $\mathcal{H}(T,T)\not\shsim H$.
        \item The links $\mathcal{H}(K_\ast,T)$ and $\mathcal{H}(U,T)$ are topologically concordant but not shake equivalent. Moreover, $\mathcal{H}(U,T)$ admits a single shake concordance to $H$.
    \end{enumerate}
\end{proposition}

\begin{proof}
All links in the statement have linking number one and hence unimodular zero-framed linking matrices. For the first assertion, Theorem \ref{thm:surgery} with $n=1$ and Lemma \ref{lem:d} give
\[
d(M_{\mathcal{H}(K,T)})
=d(S^3_1(\mathcal{D}(K)))=-2.
\]
Since $M_H\cong S^3$ and $d(S^3)=0$, Corollary \ref{cor:chain} and homology-cobordism invariance of $d$ rule out shake equivalence. Taking $K=T$, for which $\tau(T)=1$, gives the stated special case.

For the second assertion, $K_\ast$ is topologically slice, so Lemma \ref{lem:local} gives a topological concordance between $\mathcal{H}(K_\ast,T)$ and $\mathcal{H}(U,T)$. Since $\tau(K_\ast)=1$ by \eqref{eq:iterates}, the calculation above gives $d(M_{\mathcal{H}(K_\ast,T)})=-2$. On the other hand, Lemma \ref{lem:splice} gives $M_{\mathcal{H}(U,T)}\cong S^3$, so $d(M_{\mathcal{H}(U,T)})=0$. Corollary \ref{cor:chain} therefore shows that these two links are not shake equivalent.

Finally, the first component of $\mathcal{H}(U,T)$ is a positive meridian of the trefoil component. Since one crossing change unknots $T$, Bosman's construction gives a single shake concordance from $\mathcal{H}(U,T)$ to $H$ \cite[Proposition 3.1]{Bos19}.
\end{proof}

Proposition \ref{prop:examples}(i) combines Tange's surgery and correction-term calculation with the zero-surgery obstruction \cite{Tan,Bos21}. Theorem \ref{thm:main} adds a family within the topological Hopf-link class whose surgery classes are linearly independent.

%----------------------------------------------------%
\subsection{Smooth concordance classes within a shake class}
We first isolate the locality needed to pass from the knot satellite construction to strong shake concordance of links. For an ordered link $L=C_1\cup\cdots\cup C_m$ and a pattern $P\subset S^1\times D^2$, let $\mathcal S_P^{(k)}(L)$ denote replacement of the entire component $C_k$ by $P(C_k)$ in its Seifert zero-framed tubular neighbourhood, leaving all other components fixed.

% Relative satellite modification
\begin{lemma}\label{lem:relative-satellite}
    Let $P$ be a winding-number-one pattern such that $P(U)$ is smoothly slice. Then $L$ and $\mathcal S_P^{(k)}(L)$ admit a single strong shake concordance. Its surfaces on the unchanged components may be chosen to be product annuli, and its remaining surface lies in $\nu(C_k)\times[0,1]$.
\end{lemma}

\begin{proof}
The proof of \cite[Proposition 3.1]{CR} constructs a smooth connected genus-zero surface in $(S^1\times D^2)\times[0,1]$ from $P$ to a zero-framed shaking of the core. To recall the locality of that construction, write $S^1\times D^2=S^3\setminus\operatorname{int}\nu(\eta)$, where $\eta$ is the solid-torus axis. Choose a slice disk for $P(U)$ in $S^3\times[0,1]$, transverse to the product annulus $\eta\times[0,1]$. Puncture the disk at these intersections and extend the resulting boundary circles along disjoint arcs on the annulus to the opposite end, using tubes in the boundary of its normal neighbourhood. The resulting surface lies in the solid-torus product and remains connected and of genus zero. Its other boundary consists of zero-framed core parallels with algebraic count one, since $P$ has winding number one.

Embed this solid-torus product in $\nu(C_k)\times[0,1]$ using the Seifert zero framing. The surface then has one copy of $P(C_k)$ at one end and an odd number of zero-framed parallels of $C_k$, with algebraic count one, at the other. Since $\nu(C_k)$ is disjoint from the other components, this surface is disjoint from the product annuli $C_l\times[0,1]$, $l\neq k$. Together these surfaces satisfy Definition \ref{def:shake}, with all off-diagonal shaking multiplicities zero.
\end{proof}

\begin{proof}[Proof of Theorem \ref{thm:refinement}]
Fix $n\geq 1$. Let $Q$ be the winding-number-one pattern of Davis and Ray \cite[Figure 4]{DR}, and let $Q^i$ denote its $i$-fold composition, with $Q^0$ the core pattern. We use two properties of this pattern: $Q(U)=U$, and the link consisting of $Q^i$ and its solid-torus axis is topologically concordant to $H$ for every $i\geq 1$ \cite[Proposition 2.3 and the proof of Proposition 2.4]{DR}. The axis link for $Q^0$ is $H$ itself. In particular, $Q^i(U)=U$ for all $i\geq 0$.

Set
\[
    L_{n,i}=\mathcal S^{(2)}_{Q^i}(L_n),\qquad i\geq 0.
\]
Thus $L_{n,0}=L_n$, and the component knot types of $L_{n,i}$ are $\mathcal{D}(K_n)$ and $Q^i(K_\ast)$. This satellite operation is applied to the entire second component, including its passage through the Hopf clasp.

The zero-framed satellite construction is associative, with pattern composition taken using the preferred longitudes. Indeed, a zero-framed push-off of a pattern remains zero-framed under the zero-framed companion embedding. Hence, for $0\leq i<j$,
\begin{equation}\label{eq:relative-composition}
    L_{n,j}=\mathcal S_{Q^{j-i}}^{(2)}(L_{n,i})
\end{equation}
up to ordered, oriented link isotopy. This identity holds for the full satellite links, with the first component fixed. Since $Q^{j-i}$ has winding number one and $Q^{j-i}(U)=U$, Lemma \ref{lem:relative-satellite} gives a single strong shake concordance between $L_{n,i}$ and $L_{n,j}$. Thus every pair in the entire satellite family admits such a concordance directly.

To check the topological concordance type of the whole link, start with the topological concordance from $L_n$ to $H$ supplied by Theorem \ref{thm:main}. Apply the same satellite pattern $Q^i$ in a tubular neighbourhood of its second annulus. The Seifert zero framings at the two ends extend over this annulus: its oriented normal disk bundle is trivial, and a normal push-off extending the zero framing at one end has linking number zero at the other, since linking number is preserved by concordance. Satellite insertion therefore gives a locally flat concordance from $L_{n,i}$ to $\mathcal S^{(2)}_{Q^i}(H)$. The latter is the pattern--axis link $(Q^i,\eta(Q^i))$ with its component order reversed, where $\eta(Q^i)$ denotes the axis of the solid torus containing $Q^i$. By \cite[Proposition 2.3]{DR}, this link is topologically concordant to $H$ for $i\geq 1$; reversing the component order at both ends is compatible with the component-exchange symmetry of the oriented positive Hopf link. For $i=0$, the link is $H$ itself. Composing these concordances proves the assertion for every $L_{n,i}$.

It remains to select an infinite subfamily, including $L_{n,0}$, whose members are pairwise distinct in smooth concordance. We use the Legendrian satellite conventions of \cite[Section 2]{DR}. The knot $K_\ast=\mathcal{D}(T)$ has a Legendrian representative with $\operatorname{tb}=1$ and $\operatorname{rot}=0$ \cite[proof of Lemma 4.5]{CR}. Stabilizing with the sign that increases the rotation number gives a representative $\mathcal L_\ast$ with
\[
    \operatorname{tb}(\mathcal L_\ast)=0,\qquad \operatorname{rot}(\mathcal L_\ast)=1.
\]
The pattern $Q$ has a Legendrian representative with $\operatorname{tb}=2$ and $\operatorname{rot}=0$ \cite[Figure 7]{DR}. Following \cite[proof of Proposition 2.6]{DR}, stabilize this representative twice, each time increasing the rotation number, to obtain a Legendrian pattern $\mathcal Q_{\mathrm{st}}$ with $\operatorname{tb}=0$ and $\operatorname{rot}=2$. These stabilizations preserve the underlying smooth pattern. Since the winding number is one, the Legendrian satellite formulas give iterated patterns $\mathcal Q_{\mathrm{st}}^i$ with Thurston--Bennequin number zero and rotation number $2i$. The vanishing Thurston--Bennequin numbers ensure that the contact framings used in composition agree with the zero framings, so these are representatives of the untwisted iterates $Q^i$. The Legendrian satellite $\mathcal Q_{\mathrm{st}}^i(\mathcal L_\ast)$ therefore represents $Q^i(K_\ast)$ and satisfies
\[
    \operatorname{tb}\bigl(\mathcal Q_{\mathrm{st}}^i(\mathcal L_\ast)\bigr)=0,\qquad
    \operatorname{rot}\bigl(\mathcal Q_{\mathrm{st}}^i(\mathcal L_\ast)\bigr)=2i+1.
\]
The knot Floer slice--Bennequin inequality $\operatorname{tb}(\mathcal K)+|\operatorname{rot}(\mathcal K)|\leq 2\tau(K)-1$, recalled in \cite[Equation (4.1)]{CR}, now gives
\begin{equation}\label{eq:satellite-tau-growth}
    \tau\bigl(Q^i(K_\ast)\bigr)\geq i+1\qquad(i\geq 0).
\end{equation}
Together with $|\tau|\leq g_4$, this implies
\begin{equation}\label{eq:satellite-genus-growth}
    g_4\bigl(Q^i(K_\ast)\bigr)\geq i+1\qquad(i\geq 0).
\end{equation}

The lower bound \eqref{eq:satellite-tau-growth} shows that the $\tau$-values are unbounded. To obtain strictly increasing values, we select a subsequence. Since $\tau(K_\ast)=1$, set $i_0=0$ and recursively choose an integer
\[
    i_{j+1}>\max\bigl\{i_j,\tau(Q^{i_j}(K_\ast))\bigr\}.
\]
Then $0=i_0<i_1<i_2<\cdots$, and \eqref{eq:satellite-tau-growth} gives
\[
    \tau\bigl(Q^{i_{j+1}}(K_\ast)\bigr)\geq i_{j+1}+1>\tau\bigl(Q^{i_j}(K_\ast)\bigr).
\]
Set
\[
    \widehat L_{n,j}=L_{n,i_j}.
\]
Then $\widehat L_{n,0}=L_n$, and the preceding arguments give the topological concordances to $H$ and the pairwise single strong shake concordances for this subfamily.

Finally, \eqref{eq:hedden} gives $\tau(K_n)=\tau(\mathcal D(K_n))=1$ for every $n\geq 1$. Smooth link concordance induces concordances between matched components, and $\tau$ is unchanged by reversing a knot's orientation. The unordered multisets
\[
    \bigl\{1,\tau(Q^{i_j}(K_\ast))\bigr\},\qquad j\geq 0,
\]
are pairwise distinct, including the initial multiset $\{1,1\}$. They therefore distinguish the selected links even after component orderings and orientations are disregarded. Their second components have knot types $C_{n,j}=Q^{i_j}(K_\ast)$. Since $i_j\geq j$, equation \eqref{eq:satellite-genus-growth} gives
\[
    g_4(C_{n,j})\geq i_j+1\geq j+1.
\]
\end{proof}

Since every $Q^i$ has winding number one, the links $L_{n,i}$ have linking number one and hence unimodular zero-framed linking matrices. Their strong shake concordances to $L_n$ and Corollary \ref{cor:chain} therefore give
\[
    [M_{L_{n,i}}]=[M_{L_n}]\quad\text{in }\ThZ.
\]
In particular, homology-cobordism invariance of the correction term gives
\[
    d(M_{L_{n,i}})=-2\qquad(n\geq 1,\ i\geq 0).
\]
Thus the additional smooth concordance classes within each fixed shake-equivalence class share the same smooth integral homology cobordism class of zero-surgery.

%----------------------------------------------------%
\subsection{Converse questions}
Corollary \ref{cor:chain}, together with $M_H\cong S^3$, shows that shake equivalence to $H$ implies that the zero-surgery manifold is smoothly integral homology cobordant to $S^3$. We ask whether this necessary condition is sufficient within the topological concordance class of $H$.

\begin{question}\label{q:converse}
    Suppose that a two-component link $L$ is topologically concordant to $H$ and that $M_L$ is smoothly integral homology cobordant to $S^3$. Must $L\shsim H$, that is, must $L$ be connected to $H$ by a finite chain of shake concordances? More strongly, do the same hypotheses imply the existence of a single shake concordance from $L$ to $H$?
\end{question}

For locally knotted Hopf links, we can also ask the analogous question without assuming topological concordance to $H$.

\begin{question}\label{q:local}
    For arbitrary knots $J,K$, does $[M_{\mathcal{H}(J,K)}]=0$ in $\ThZ$ imply $\mathcal{H}(J,K)\shsim H$?
\end{question}

The homology-cobordism condition in these questions is equivalent to the zero-surgery manifold bounding a smooth integral homology $4$-ball. Vanishing of the correction term alone does not establish this condition.

For $\mathcal H(J,K)$, topological concordance to $H$ is equivalent to topological sliceness of both $J$ and $K$. Indeed, restricting a concordance to the individual component annuli gives necessity, while Lemma \ref{lem:local} gives sufficiency. Proposition \ref{prop:examples}(ii) provides an example of shake equivalence to $H$ outside its topological concordance class, namely $\mathcal H(U,T)$, motivating the broader formulation of Question \ref{q:local}.

%----------------------------------------------------%
\subsection{Nonzero framings}
Bosman's definition also allows nonzero integer framings \cite[Section 2]{Bos19}. With a common integer framing $r$ on the Hopf components, the surgery matrix is
\[
    \Lambda_r=
    \begin{pmatrix}
        r&1\\
        1&r
    \end{pmatrix},
    \qquad \det\Lambda_r=r^2-1.
\]
It is unimodular only for $r=0$. Thus a nonzero common framing takes the Hopf surgery outside the integral homology-sphere setting used here. Nonzero-framing analogues would require compatible framed surgery identities and a homology-cobordism obstruction suited to the resulting surgery manifolds.

\vspace{1em}
\paragraph{Tool Use} 
AI tools were used to assist with language editing and checks of references.

\bibliographystyle{amsalpha}
\bibliography{references}

\end{document}